\documentclass[12pt,letterpaper]{amsart} 

\usepackage{amsmath, amsfonts, amsthm, amssymb}
\usepackage{exscale}
\usepackage{cite}
\usepackage{epsfig}
\usepackage{amscd}
\usepackage{graphics}
\usepackage{color}
\usepackage{dsfont}

\renewcommand{\geq}{\geqslant}
\renewcommand{\leq}{\leqslant}

\newtheorem{theorem}{Theorem}
\newtheorem{conjecture}[theorem]{Conjecture}

\newtheorem{proposition}{Proposition}[section]

\newtheorem{corollary}[proposition]{Corollary}
\newtheorem{lemma}[proposition]{Lemma}

\newtheorem*{main-theorem}{Main Theorem}
\newtheorem*{theorem*}{Theorem}

\theoremstyle{definition}

\newtheorem{remark}[proposition]{Remark}

\newtheorem*{remark*}{Remark}

\numberwithin{equation}{section}

\def\phi{\varphi}

\newcommand{\HS}{{\mathbb H}}
\newcommand{\bbH}{{\mathbb H}}
\newcommand{\bbR}{{\mathbb R}}
\newcommand{\bbC}{{\mathbb C}}

\newcommand{\RR}{\mathbb{R}}

\newcommand{\abs}[1]{\left\lvert#1\right\rvert}
\newcommand{\norm}[1]{\left\Vert#1\right\Vert}
\newcommand{\brak}[1]{\left\langle#1\right\rangle}
\newcommand{\re}{\operatorname{Re}}
\newcommand{\im}{\operatorname{Im}}

\def\calH{\mathcal H}

\def\Im{\,\mathrm{Im}\,}

\def\supp{\mathrm{supp}\,}

\def\phi{\varphi}

\def\be{\begin{eqnarray*}}
\def\ee{\end{eqnarray*}}
\def\ben{\begin{eqnarray}}
\def\een{\end{eqnarray}}

\def\L2R{L_{\text{Rest}}^2}

\def\11{\mathds{1}}

\def\L2c{L^2_{\text{comp}}}

\def\11{\mathbb{1}}
\def\Vol{\text{Vol}}

\newcommand{\del}{\partial}
\newcommand{\cinf}{C^\infty}

\begin{document}

\title[Dispersive Estimates]{Dispersive Estimates for Scalar and Matrix Schr\"odinger operators on $\HS^{n+1}$}

\author[D. Borthwick]{David Borthwick}
\email{borthwick@math.emory.edu}
\address{Department of Mathematics, Emory University}

\author[J.L. Marzuola]{Jeremy L. Marzuola}
\email{marzuola@math.unc.edu}
\address{Department of Mathematics, UNC-Chapel Hill \\ CB\#3250
  Phillips Hall \\ Chapel Hill, NC 27599}

\subjclass[2000]{}

\begin{abstract}

We study resolvent estimates, spectral theory and large time dispersive properties of scalar and matrix Schr\"odinger-type operators on $\HS^{n+1}$ for $n \geq 1$.

\end{abstract}

\maketitle

\section{Introduction}
\label{sec:intro}

In this note, we explore the dispersive behavior of solutions to the perturbed Schr\"odinger equation on hyperbolic space,
\begin{equation}
i\partial_t u - \Delta u  + V u = 0,  
\label{eqn:hs}
\end{equation}
where $\Delta$ is the (non-positive definite) Laplacian on $\bbH^{n+1}$, $n\ge 1$, and $V$ is a real potential.  
We will also consider certain matrix versions of this equation, motivated by stability questions for the non-linear Schr\"odinger equation in $\bbH^{n+1}$.

Embedded eigenvalues and/or resonances would present obstructions to dispersive estimates, 
but in the scalar case we can rule these out under mild decay assumptions on the potential, except at the bottom of the continuous spectrum.  The decay condition is express in terms of the function $\rho := e^{-r}$, where $r$ denotes the radial geodesic coordinate on $\bbH^{n+1}$.
(For the conformal  compactification of $\bbH^{n+1}$, $\rho$ serves as a boundary-defining coordinate.)

The free resolvent on $\bbH^{n+1}$ is usually written in the form 
\[
R_0(s) := (-\Delta - s(n-s))^{-1},
\]
with the half-plan $\{\re s > \tfrac{n}2\}$ corresponding to the resolvent set of $-\Delta$.   
The critical line $\{\re s = \tfrac{n}2\}$ is a double cover of the continuous spectrum, $\sigma(-\Delta) = [\tfrac{n^2}4, \infty)$.  
With this convention, $R_0(s)$ admits a meromorphic continuation to $s \in \bbC$, as a bounded operator $\rho^{N} L^2(\bbH^{n+1}) \to \rho^{-N} L^2(\bbH^{n+1})$ for $\re s > \tfrac{n}2 - N$.

\begin{theorem}\label{absence.thm}
For $V \in \rho^\alpha L^{\infty}(\bbH^{n+1}, \bbR)$ with $\alpha>0$, the operator $-\Delta + V$ has continuous spectrum $[\tfrac{n^2}4, \infty)$, with no embedded eigenvalues in the range $(\tfrac{n^2}4, \infty)$.  Moreover, the resolvent,
\[
R_V(s) := (-\Delta + V - s(n-s))^{-1},
\] 
admits a meromorphic continuation to $\re s \ge \tfrac{n}2 - \delta$ as an operator $\rho^\delta L^2 \to \rho^{-\delta}L^2$, for 
$\delta < \alpha/2$.  The continued resolvent $R_V(s)$ has no poles 
on the critical line $\re s = \tfrac{n}2$ except possibly at $s = \tfrac{n}2$.
\end{theorem}

For smooth potentials, an eigenvalue at $n^2/4$ (the bottom of the continuous spectrum), is ruled out by Bouclet \cite{bouclet2013absence} under a weaker decay assumption.   However, we are not aware of any condition on $V$ that would rule 
out a resonance at $s = \tfrac{n}2$, so we will take the regularity of $R_V(s)$ at this point as an assumption.

\begin{theorem}
\label{thm1}
Suppose $V\in \rho^\alpha L^{\infty}(\bbH^{n+1}, \bbR)$ with $\alpha>0$ and
\begin{equation}\label{alpha.cond}
\frac{\alpha}n > 1 - \left\lfloor \frac{n+5}4 \right\rfloor^{-1}.
\end{equation}
Assuming that $R_V(s)$ does not have a pole at $s = \tfrac{n}2$,  for $t \geq 1$ we have the dispersive bound
\[
\norm{e^{i t (-\Delta + V)} P_c }_{L^1 \to L^\infty} \le C_n  t^{-\frac32},
\]
where $P_c$ denotes the projection on the continuous spectrum of the operator $-\Delta + V$.  
\end{theorem}

The proof involves estimation of the kernel of $R_V(s)$ by applying a version of Young's inequality to terms in the Birman-Schwinger resolvent expansion.  The restriction on $\alpha$ results from the $L^p$ estimates of the free resolvent kernel used in this technique and may not be sharp.  In this result we focus on the long-time behavior, which corresponds to the low-frequency asymptotic behavior of the perturbed spectral measure.  The short-time estimate would be an inherently high-frequency result and would require a very careful construction of a semiclassical parametrix for the perturbed operator.  For rough and/or long range potentials the standard semiclassical methods do not apply directly.  As a result, we will restrict our analysis of perturbed Laplacian operators to large times throughout the paper.  

Dispersive estimates of this type are motivated by trying to generalize the notion of wave operators for the 
Schr\"odinger equation on $\HS^{n+1}$,
and also by the question of asymptotic stability of nonlinear bound states in $\HS^{n+1}$.  To see how linearization at a bound state gives rise to a matrix equation, consider a general NLS equation of the form,
\[
i\partial_t u + \Delta u + \beta(\abs{u}^2) u = 0.  
\]
The bound states in question are solutions of the form,
\begin{equation*}
u(t,z) = e^{i (\mu - \frac{n^2}4) t} \Psi(z) ,
\end{equation*} 
where $\mu > \tfrac{n^2}4$ and $\Psi$ is a solution for the corresponding stationary problem,
\begin{equation}
\label{eqn:stat}
-\Delta \Psi + (\mu - \tfrac{n^2}4) \Psi - \beta(\abs{\Psi}^2) \Psi  = 0.
\end{equation}
(We shift the parameter $\mu$ to account for the fact that the spectrum of $-\Delta$ starts at $\tfrac{n^2}4$.)
For the polynomial case $\beta(|u|^2) = |u|^{p}$, existence of such bound states is established for $0< p < 4/(n-1)$ in \cite{CMMT,ChMa-hyp,ManSan}.  Furthermore, these bound state solutions are shown to be radial and positive.

To linearize at the bound state, we take the ansatz
\begin{equation}\label{perturbed.soln}
u(t,x) = e^{i (\mu - \frac{n^2}4) t}(\Psi(z) + \phi(t,z)).
\end{equation}
Inserting this into the NLS equation, and using \eqref{eqn:stat} to simplify, we have 
\[
i \partial_t \varphi + \Delta \varphi - (\mu - \tfrac{n^2}4) \varphi + \beta(\abs{\Psi}^2)\varphi + 2\beta'( \abs{\Psi}^2 ) \Psi^2 \re(\varphi) 
= O(\varphi^2),
\]
The presence of the term $\re \phi$ turns this into a system of the form,
\begin{equation}\label{eqn:mhs}
(i\partial_t + \calH) \begin{pmatrix} \varphi \\ \overline{\varphi} \end{pmatrix} = 0,
\end{equation}
with matrix Schr\"odinger operator,
\begin{equation}\label{calH.def}
\calH := \begin{pmatrix} -\Delta + (\mu - \tfrac{n^2}4) & 0 \\
0 & \Delta - (\mu - \tfrac{n^2}4)  \end{pmatrix}
+ \begin{pmatrix} -V_1  & -V_2 \\
V_2 & V_1  \end{pmatrix}.
\end{equation}

Since the $V_j$'s are combinations of $\beta(\abs{\Psi}^2)$ and $\beta'(\abs{\Psi}^2) \Psi^2$, they inherit decay and regularity properties from the bound state solutions $\Psi$.  
Following the ideas, for instance, of \cite{Kwong} and \cite{PelSer}, one can show for $\beta(\abs{u}^2) = \abs{u}^p$ 
that the radial solutions $\Psi$ satisfy $\Psi \in \rho^{n/2 + \sqrt{\mu} - \epsilon}L^\infty$ for $\epsilon>0$.  It follows in this case that the potentials satisfy $V_1, V_2 \ \in \rho^\alpha L^{\infty}(\bbH^{n+1})$ for $\alpha < p(\tfrac{n}2 + \sqrt{\mu})$.  For further regularity and decay properties of bound states, see \cite{CMMT}.  

Note that by an application of Weyl's Theorem \cite[Theorem XIII.14]{RSv4}, under the assumption that $V_1, V_2 \ \in \rho^\alpha L^{\infty}(\bbH^{n+1})$, the continuous spectrum of the operator \eqref{eqn:mhs} is $(-\infty, - \mu] \cup [ \mu, \infty)$.  
Because $R_0(s)V_j$ is a compact operator on $L^2(\bbH^{n+1})$ for $\re s \ge \tfrac{n}2$, 
the argument follows verbatim from the Euclidean space argument in \cite[Lemma 3]{ES1}.
In fact, via the symmetry properties of $\calH$, one can observe that the spectrum must be contained in the union of the real and imaginary axes.  We will not we this fact here, because we are only concerned with the projection onto the continuous spectrum.

We will prove the following theorem for solutions to \eqref{eqn:mhs}.
\begin{theorem}
\label{thm2}
Let $n \geq 1$ and $V_1, V_2 \ \in \rho^\alpha L^{\infty}(\bbH^{n+1}, \bbR)$ with $\alpha>0$ satisfying \eqref{alpha.cond}.  Assume that 
$\mathcal{H}$ has no embedded or endpoint eigenvalues or resonances.  Then, for $t \geq 1$
\[
\norm{e^{-it \calH } P_c }_{L^1 \to L^\infty} \le C_d  t^{-\frac32},
\]
where $P_c$ denotes the projection on the continuous spectrum of $\calH$.   
\end{theorem}

\begin{remark}
Following the standard method, the above dispersive estimates will give an extended range of Strichartz estimates the perturbed hyperbolic Scrh\"odinger problem, as in for instance \cite{Ban1,BCS,anker2009nonlinear,ionescu2012global}.
\end{remark}

For the scalar case we were able to establish absence of embedded eigenvalues and resonances using properties of the potential.  
Since $\calH$ is not self-adjoint, it is more difficult to rule these out in the matrix case.  
For methods of verification of these spectral conditions for the matrix operators in $\RR^d$, see \cite{M-Simpson}.  
Further analysis of the spectrum of $\calH$ will be a topic of future work towards the asymptotic stability question. 

In  \cite{S1}, Schlag studies the behavior of solutions near a nonlinear bound state for the cubic Schr\"odinger equation on $\RR^3$ and introduces 
a strong notion of stability.   The dispersive estimates of Theorem~\ref{thm2} constitute a crucial component for asymptotic stability analysis of a similar form in $\HS^{n+1}$.  

It is proved in \cite{S1} that there is is a codimension-one stable manifold of perturbations to the ground state for the cubic Euclidean NLS equation in $\RR^3$.  In $\RR^2$, the cubic NLS is $L^2$-critical and hence all possible bound states have the same $L^2$ mass, from a scaling argument, and display self-similar blow-up; see for instance \cite{merle2005blow} and many others referenced within.  The soltions for $\HS^2$ and $\HS^3$ can be seen to be orbitally unstable as in the recent work \cite{banica2014global}.
We note that blow-up is known to occur for mass above that of the Euclidean ground state by an argument in \cite{Ban1} using arguments of Glassey in \cite{Glassey}.  Since proving this requires a much more detailed analysis for the components of $\calH$ in the $\HS^2$ setting, we state it here as a conjecture.  See also \cite{raphael2011existence} for a related problem regarding an inhomogeneous cubic NLS equation in $\RR^2$.  In $\RR^3$, the cubic NLS is supercritical.  However, the spectral properties of the matrix operator are simpler due to the lack of scaling invariance in $\HS^{n+1}$.  

\begin{conjecture}
In $\HS^2$ and $\HS^3$, there is a codimension-one manifold of stable perturbations of a soliton for the cubic NLS equation.  The instability will related to an unstable direction in the spectrum of the operator linearized about the soliton.  In general, orbitally stable bound states associated with $C^2$ nonlinearities are actually asymptotically stable.
\end{conjecture}

We plan to address these questions of long time dynamics in future work.
This investigation will depend upon spectral properties of the linearized operator about a soliton in $\HS^{n+1}$ and on subsequent stability results.  

The paper is organized as follows.  In Section \ref{sec:scalarhom}, we offer a different proof of the dispersive estimates for the free Schr\"odinger equation for $t \geq 1$,highlighting a simple and elegant treatment of the resolvent in $\HS^{n+1}$.  Such estimates have previously appeared in for instance the works \cite{Ban1,BCS,bcd2009,pierfelice2008weighted,IS,anker2009nonlinear,bouclet2011strichartz,ionescu2012global} and others, most of which also analyze interesting behaviors like scattering or blow-up for nonlinear Schr\"odinger equations on $\HS^{n+1}$.  
In Section \ref{sec:specasym} we develop the asymptotic properties of the free resolvent kernel, in various frequency and spatial limits.  Then, we will develop the necessary operator norm estimates for the full (scalar) resolvent in Section \ref{spectheory}.  
We prove Theorem~\ref{absence.thm} in Section \ref{absence}, using Carleman-style estimates.  Finally, in Sections \ref{sec:scalarinhom} and \ref{sec:matrix}, we will analyze the dispersive properties of both the inhomogeneous scalar and matrix Schr\"odinger equations and prove Theorems~\ref{thm1} and \ref{thm2}, respectively.  

The work here can be seen as an extension of the scattering theory to perturbations of the hyperbolic Laplacian, similar to the theory developed for the perturbed scalar and matrix Euclidean Schr\"odinger equations in \cite{RSv4} as well as more recently in \cite{GolSch,ES1,S1,Mar-lin} and many others.  As noted above, the matrix operator involves a non-self-adjoint perturbation of a self-adjoint matrix operator.  In that regard, it would be interesting to treat non-self-adjoint perturbations of the scalar problem as well, such as $-\Delta_{\HS^{n+1}} + W \nabla_{\HS^{n+1}} + V$, though more spectral assumptions would be required in that case and we do not treat it here.  

\section*{Acknowledgments} DB received support from NSF Grant DMS--0901937.  JLM was supported by NSF Grant DMS--1312874 and NSF CAREER Grant DMS--1352353.  The authors are grateful to the organizers of the ``Quantum chaos, resonances and semi-classical measure'' program in Roscoff, France, where this work began.  In addition, JLM is grateful to the Hausdorff Institute of Mathematics in Bonn, Germany where part of this work was completed, as well as for several useful conversations over the last several years with Michael Taylor, especially with respect to the decay of nonlinear bound states on $\HS^{n+1}$.  In addition, he thanks Hans Christianson, Jason Metcalfe, Enno Lenzmann and Thomas Boulenger for many useful discussions on geometric scattering theory.  We thank the anonymous referee for making suggestions to improve an earlier version of the draft.

\section{Dispersive estimates for the free Laplacian in $\bbH^{n+1}$}
\label{sec:scalarhom}

To motivate the treatment of the perturbed case later on, we will first present a proof 
of the large-time $L^1 \to L^\infty$ dispersive bound for the free Schr\"odinger equation in $\bbH^{n+1}$.  
Note that for this free case both large and small time bounds have been proven by somewhat different 
approaches in several other references, such as \cite{Ban1,pierfelice2008weighted,IS,ionescu2012global}.  
Our goal here is to highlight the fact that it is the degree of vanishing of the spectral resolution at the bottom of the spectrum, plus smoothness in the spectral parameter, that gives rise to the power $t^{-3/2}$ in the large-time dispersive estimate.  One can see clearly from the proof that a cutoff to high-frequencies would yield a large-time decay of order $t^{-\infty}$.  

\begin{proposition}\label{prop:freedisp}
For $g \in L^1(\bbH^{n+1})$,
\[
\abs{e^{-it\Delta}g(z)} \le C_n t^{-\frac32} 
\int_{\bbH^{n+1}} \brak{d(z,w)}^{n+1} e^{-\frac{n}2 d(z,w)} \abs{g(w)}\>dV(w).
\]
In particular, for $t \geq 1$
\[
\norm{e^{-it\Delta}}_{L^1 \to L^\infty} \le C_nt^{-\frac32}.
\]
\end{proposition}

By standard convention the resolvent of the Laplacian $-\Delta$ on $\bbH^{n+1}$ is written,
\[
R_0(s) := (-\Delta - s(n-s))^{-1},
\]
with $\re s > \tfrac{n}2$ corresponding to the resolvent set $s(n-s) \in \mathbb{C} - [\tfrac{n^2}4,\infty)$.  
The choice of $s$ as a spectral parameter is motivated by the hypergeometric formula for the kernel,
\[
R_0(s;z,w) =  \frac{\pi^{-\frac{n}2} 2^{-2s-1}\Gamma(s)}{\Gamma(s-\frac{n}2+1)} \cosh^{-2s}(\tfrac{r}2) 
F(s,s-\tfrac{n-1}2, 2s-n+1; \cosh^{-2}(\tfrac{r}2)),
\]
where $r := d(z,w)$.  In this case, $F$ is the representation of the standard hyergeometric function.
If we define $\nu := s - \tfrac{n+1}2$ and $\mu := \frac{n-1}2$, then this could also be written in terms of a Legendre function
\begin{equation}\label{R0.legendre}
R_0(s;z,w) = (2\pi)^{-\frac{n+1}2} e^{-i\pi\mu} (\sinh r)^{-\mu} Q_\nu^\mu(\cosh r).
\end{equation}
For convenience, we will use these assignments for $\nu$ and $\mu$ in all of the Legendre function formulas.

With the hyperbolic convention for the spectral parameter, 
Stone's formula gives the continuous part of the spectral resolution as
\[
d\Pi(\lambda) := 2i\lambda \left[R_0(\tfrac{n}2 + i\lambda) - R_0(\tfrac{n}2 - i\lambda)\right] \>d\lambda.
\]
Up to a simple factor, the kernel of the spectral resolution is thus given by $\im R_0(\tfrac{n}2 + i\lambda; z,w)$. 
By the Legendre connection formula,
\[
Q_{-\nu-1}^\mu(z) - Q_{\nu}^\mu(z) = e^{i\pi\mu} \cos(\pi \nu) \Gamma(\mu+\nu+1) \Gamma(\mu-\nu) P_{\nu}^{-\mu}(z),
\]
we have, for $\lambda \in \bbR$,
\begin{equation}\label{imR0}
R_0(\tfrac{n}2  + i\lambda;z,w) - R_0(\tfrac{n}2 - i\lambda;z,w) :=  A_n(\lambda) (\sinh r)^{-\mu} P_{\nu}^{-\mu}(\cosh r),
\end{equation}
where 
\begin{equation}\label{An.def}
A_n(\lambda) := c_n \abs{\Gamma(\tfrac{n}2 + i\lambda)}^2 \sinh (\pi \lambda).
\end{equation}
Note that by Stirling's formula, $A_n(\lambda) = O(\lambda^{n-1})$.  

The small-time estimate from \cite{Ban1} or \cite{IS} has the same form as Proposition~\ref{prop:freedisp}, 
except that the power on the right-hand side is $t^{-(n+1)/2}$.  This essentially follows from high-frequency estimates on the 
representation \eqref{imR0} which we will state more explicitly in the next section.

To develop the proof of Proposition~\ref{prop:freedisp}, for $g \in L^1(\bbH^{n+1})$, we can use the spectral resolution to write
\begin{equation}
\label{eqn:specrep_hom}
e^{-it\Delta}g(z) = \frac{e^{in^2/4}}{2\pi i} \int_{-\infty}^\infty \int_{\bbH^{n+1}} e^{i t \lambda^2} g(w) 
\> d\Pi(\lambda; z,w)\>dV(w).
\end{equation}
The novelty in our approach to the free case is the use of a particular formula for the Legendre function from \cite[\S3.7, eq.~(8)]{Erdelyi},
\begin{equation}\label{legP.rep}
P_{\nu}^{-\mu}(\cosh r) = \sqrt{\frac{2}{\pi}} \frac{(\sinh r)^{-\mu}}{\Gamma(\mu+\frac12)} 
\int_0^r (\cosh r - \cosh u)^{\mu-\frac12} \cos(\lambda u)\>du,
\end{equation}
valid for $\mu > -\tfrac12$ and $\lambda \in \bbR$.
In view of the representation \eqref{legP.rep}, we introduce the kernel
\begin{equation}
\label{Kdef}
K(u;r) := (\sinh r)^{-2\mu} (\cosh r - \cosh u)^{\mu-\frac12} \chi_{[0,r]}(u).
\end{equation}
\begin{lemma}\label{ghz.lemma}
For $g \in L^1(\bbH^{n+1})$ and $z \in \bbH^{n+1}$, set
\[
h_z(u) := \int_{\bbH^{n+1}} K(u; d(z,w)) g(w)\>dV(w).
\]
Then for any $k \ge 0$,
\[
\int_0^\infty \brak{u}^k \abs{h_z(u)}\>du \le c_n \int_{\bbH^{n+1}} \brak{d(z,w)}^{k+1} e^{-\frac{n}2 d(z,w)} \abs{g(w)}\>dV(w). 
\]
\end{lemma}
\begin{proof}
The kernel $K(u;r)$ is comparable to $r^{-1} \chi_{[0,r]}(u)$ near $r=0$ and exponentially decreasing as $r \to \infty$.  Hence, 
for $g \in L^1(\bbH^{n+1})$ we can apply Fubini to compute
\[
\begin{split}
&\int_0^\infty \brak{u}^k \abs{h_z(u)}\>du \\
&\qquad= c_n\int_0^\infty (\sinh r)^{-2\mu}  \int_0^r (\cosh r - \cosh u)^{\mu-\frac12} \brak{u}^k 
\abs{\tilde g_z(r)} \sinh^n r\>du\>dr,
\end{split}
\]
where $\tilde g_z(r)$ denotes the average of $g(w)$ over a sphere of radius $r$ centered at the point $z$.  
We can then simply use the restriction $u\in [0,r]$ and \eqref{legP.rep} to estimate
\[
\int_0^\infty \brak{u}^k \abs{h_z(u)}\>du \le c_n\int_0^\infty \brak{r}^k (\sinh r)^{-\mu}  P^{-\mu}_{-\frac12}(\cosh r) \abs{\tilde g_z(r)} \sinh^n r \>dr.
\]
The function $(\sinh r)^{-\mu} P^{-\mu}_{-\frac12}(\cosh r)$ is regular at $r=0$ and has the asymptotic
\[
(\sinh r)^{-\mu} P^{-\mu}_{-\frac12}(\cosh r) \sim c_n r e^{-nr/2}\quad\text{as }r\to \infty.
\]
This yields
\[
\int_0^\infty \brak{u}^k \abs{h_z(u)}\>du \le c_n\int_0^\infty \brak{r}^{k+1} e^{-nr/2} \abs{\tilde g_z(r)} \sinh^n r \>dr,
\]
and the result follows.
\end{proof}

In terms of the function $h_z(u)$ introduced in Lemma~\ref{ghz.lemma}, we can now use \eqref{imR0} and \eqref{legP.rep} to 
rewrite \eqref{eqn:specrep_hom} as
\[
e^{-it\Delta}g(z) = c_n \int_{-\infty}^\infty \int_0^\infty e^{it\lambda^2} \lambda A_n(\lambda) \cos(\lambda u) h_z(u)\>du\>d\lambda,
\]
where $h_z$ is the function introduced in Lemma~\ref{ghz.lemma}.  For convenience, let us extend $h_z$ to $u <0$ as an even function,
$h_z(-u) := h_z(u)$, so that we can write this as
\begin{equation}\label{A.hz}
e^{-it\Delta} g(z) = c_n \int_{-\infty}^\infty e^{it\lambda^2}  \lambda A_n(\lambda) \hat{h}_z(\lambda)\>d\lambda.
\end{equation}
The coefficient $A_n(\lambda)$ is essentially a polynomial, so to handle this expression we first prove a lemma that illustrates how 
powers of $\lambda$ in the spectral resolution translate into decay in $t$.  In fact, we can simplify
the formula \eqref{An.def} for $A_n(\lambda)$ to 
\[
A_n(\lambda) = 
\begin{cases}
\bigl[a_{n}\lambda^{n} + \dots + a_1 \lambda \bigr] \tanh(\pi\lambda)& n\text{ odd},\\
a_{n}\lambda^{n} + \dots + a_2 \lambda^2 & n\text{ even}.
\end{cases}
\]

\begin{lemma}\label{disp.h}
If $h$ is a bounded smooth function with bounded derivatives, then for $h(u) \in \brak{u}^{-k} L^1(\bbR)$, 
we have for $t \geq 1$
\[
\abs{\int_{-\infty}^\infty e^{it\lambda^2} \lambda^k \hat{h}(\lambda)\>d\lambda} 
\le c_k t^{-\lfloor \frac{k+1}2 \rfloor-\frac12} \int_{-\infty}^\infty \brak{u}^k \abs{h(u)}\>du.
\]
\end{lemma}
\begin{proof}
It suffices to consider a Schwarz function $h \in \mathcal{S}$.  Note that for $k=0$ the standard dispersive estimate for the free 
Schr\"odinger equation in $\bbR$ gives
\begin{equation}\label{disp.k0}
\abs{\int_{-\infty}^\infty e^{it\lambda^2} \hat{h}(\lambda)\>d\lambda} \le Ct^{-\frac12} \norm{h}_{L^1}.
\end{equation}
Similarly, if $k=1$ we can integrate by parts once to obtain
\[
\int_{-\infty}^\infty e^{it\lambda^2} \lambda \hat{h}(\lambda)\>d\lambda = 
- \frac{1}{2it} \int_{-\infty}^\infty e^{it\lambda^2} \del_\lambda\hat{h}(\lambda)\>d\lambda.
\]
And then, since $\mathcal{F}^{-1}(\hat h')(u) = -iu h(u)$, the dispersive bound \eqref{disp.k0} gives
\begin{equation}\label{disp.k1}
\abs{\int_{-\infty}^\infty e^{it\lambda^2} \lambda \hat{h}(\lambda)\>d\lambda}
\le Ct^{-\frac32} \norm{uh(u)}_{L^1}.
\end{equation}

For $k \ge 2$ integration by parts gives
\[
\int_{-\infty}^\infty e^{it\lambda^2} \lambda^k \hat{h}(\lambda)\>d\lambda = 
- \frac{1}{2it} \int_{-\infty}^\infty e^{it\lambda^2} \left[ (k-1)\lambda^{k-2} \hat{h}(\lambda) 
+ \lambda^{k-1} \del_\lambda\hat{h}(\lambda) \right]\>d\lambda.
\]
By iterating this formula we can reduce to a combination of the cases \eqref{disp.k0} or \eqref{disp.k1}.  
\end{proof}

%

\bigbreak
\begin{proof}[Proof of Proposition~\ref{prop:freedisp}]
The proof essentially follows from applying Lemma~\ref{disp.h} to \eqref{A.hz}.   For $n$ odd we define $f := \mathcal{F}^{-1}(\tanh(\pi\lambda)/\lambda)$, in the distributional sense.
Analyticity implies that $f$ is represented by an integrable function with exponential decay. 
Hence the map $h \mapsto f*h$ is bounded as a map
$\brak{u}^{-k} L^1(\bbR) \to \brak{u}^{-k} L^1(\bbR)$ for any $k$.   

Thus, in any dimension, we can apply Lemma~\ref{disp.h} to the polynomial terms in $\lambda A_n(\lambda)$.  
The $\lambda^2$ term fixes  the leading $t^{-\frac32}$ decay rate for large $t$, while the higher degree terms require additional decay of the function $h_z$.  The result is a pointwise bound,
\[
\abs{e^{-it\Delta}g(z)}  \le C_n t^{-\frac32} \int_{0}^\infty \brak{u}^n \abs{h_z(u)}\>du
\]
for large times.  An application of Lemma~\ref{ghz.lemma} completes the proof.
\end{proof}

\section{Free resolvent kernel estimates}
\label{sec:specasym}

To handle the case of $-\Delta + V$, we will need pointwise estimates on the free resolvent kernel.
For the scalar case, we only need consider $R_0(\tfrac{n}2 + \sigma)$ with $\sigma$ purely imaginary, but in the matrix case we will also need estimates positive real $\sigma$, so we will treat the general case $\re\sigma \ge 0$ below.

\begin{lemma}\label{PQ.lemma}
For $\mu$ fixed and $\arg \sigma \in [0,\tfrac{\pi}2]$, the Legendre functions can be estimated in terms of modified Bessel functions, 
\begin{equation}\label{PI.est}
P_{-\frac12 + \sigma}^{-\mu}(\cosh r) = \sigma^{-\mu} \left(\frac{r}{\sinh r}\right)^{\frac12} I_\mu(\sigma r) (1 + O_\mu(\sigma^{-1})),
\end{equation}
and
\begin{equation}\label{QK.est}
Q_{-\frac12 + \sigma}^{\mu}(\cosh r) =  e^{i\pi\mu} \sigma^{\mu} \left(\frac{r}{\sinh r}\right)^{\frac12} K_\mu(\sigma r) (1 + O(\sigma^{-1})),
\end{equation}
both uniformly for $r \in (0,\infty)$ for $\arg \sigma \in [0,\pi/2]$.  
\end{lemma}
\begin{proof}
For the case of $\sigma \in (0,\infty)$, this is proven in \cite[\S12.12.3]{Olver}.  The estimate for the full range 
$\arg \sigma \in [0,\pi/2]$ essentially follows from the same approach, so we will only sketch the details.  

With either $L = \sigma^\mu P_{-\frac12 + \sigma}^{-\mu}(\cosh r)$ or 
$L = \sigma^{-\mu} Q_{-\frac12 + \sigma}^{\mu}(\cosh r)$, we set
\[
W = (r \sinh r)^\frac12 L, \qquad \zeta := r^2.
\]
The Legendre equation transforms into 
\[
\frac{d^2 W}{d\zeta^2} = \left( \frac{\sigma^2}{4\zeta} + \frac{\mu^2 - 1}{4\zeta^2} + \frac{\psi}{\zeta} \right) W,
\]
which is almost a Bessel equation except for the error term
\[
\psi(r) := \frac{4\mu^2-1}{16} \left( \frac{1}{\sinh^2 r} - \frac{1}{r^2} \right).
\]

Let us now specialize to the $P$ case.  If make the anszatz,
\[
W = \sigma^\mu (r \sinh r)^\frac12 P_{-\frac12 + \sigma}^{-\mu}(\cosh r) = r I_\mu(\sigma r) + r_\mu(\sigma,r),
\]  
then using the equation for $W$ and the boundary conditions appropriate to the $P$-solution, we can derive a recursive 
integral equation for the error term (see \cite[eq.~(12.03.08)]{Olver},
\[
r_\mu(\sigma,r) = 2r \int_0^r \Bigl[I_\mu(\sigma r) K_\mu(\sigma r) - K_\mu(\sigma r) I_\mu(\sigma t) \Bigr] \psi(t)
\Bigl(r_\mu(\sigma,t) + 2tI_\mu(\sigma t)\Bigr) dt.
\]
The key properties of $\psi$ that lead to an estimate on the error term are that $\psi$ is monotonic and integrable over $[0,\infty)$. Beyond this, we only need to use well known properties of the modified Bessel functions.  For $\mu >0$ as $z \to 0$
\begin{equation}\label{IK.zero}
I_\mu(z) \sim \frac{1}{\Gamma(\mu+1)} \left(\frac{z}2 \right)^\mu, \qquad K_\mu(z) \sim \frac12 \Gamma(\mu) \left(\frac{z}2 \right)^{-\mu}.
\end{equation}
For $\mu >0$ as $z \to \infty$, we have
\begin{equation}\label{IK.infty}
I_\mu(z) \sim (2\pi z)^{-\frac12} \left( e^z + ie^{\mu \pi i} e^{-z}\right), \qquad
K_\mu(z) \sim  \left(\frac{\pi}{2z}\right)^{\frac12} e^{-z},
\end{equation}
valid for $\arg z \in [0,\tfrac{\pi}2]$.  Using these estimates, together with the result in \cite[Thm.~12.3.1]{Olver}, we obtain a bound for the error term,
\[
r_\mu(\sigma,r) \le C_\mu r I_\mu(\sigma r) \sigma^{-1},
\]
which yields \eqref{PI.est}.

The same approach applies for the $Q$ solution.  The only notable difference in the argument is that the boundary 
conditions on the ansatz are applied at $r = \infty$ in this case.
\end{proof}

We note here that applying \eqref{PI.est} and \eqref{IK.infty} to \eqref{imR0} gives a high-frequency 
asymptotic formula for the spectral resolution, with a leading order of $\lambda^{(n+1)/2}$.  
This asymptotic is the essential input for the proof of the small-$t$ version of Proposition~\ref{prop:freedisp}, 
for example via the dyadic partition argument used in Lemma~\cite[Lemma~3.2]{IS}.  
To extend the small-$t$ results to the perturbed case would require corresponding results for
the perturbed spectral resolution for large $\lambda$.

From the uniform asymptotics of the Legendre functions we can derive pointwise bounds on the resolvent kernel and spectral resolution.  This bounds will be crucial for the dispersive estimates.  
\begin{corollary}\label{R0.cor}
For the free resolvent kernel we have the pointwise bounds,
\begin{equation}\label{R0.kernel}
\abs{R_0(\tfrac{n}2 + \sigma, z,w)}  \le \begin{cases} 
C \abs{\log r} & \abs{r\sigma} \le 1, n=1, \\
C_{n}  r^{1-n} & \abs{r\sigma} \le 1, n\ge 2, \\
C_{n} \abs{\sigma}^{n/2-1} e^{-(\frac{n}2 + \re \sigma) r} & \abs{r\sigma} \ge 1, \end{cases}
\end{equation}
where $r:= d(z,w)$, valid for $\re \sigma \ge 0$, $\abs{\sigma} \ge 1$, and $r \in (0,\infty)$.  
For derivatives with respect to $\sigma$ we have, for any $\epsilon >0$,
\begin{equation}\label{dR0.kernel}
\abs{\del_\sigma^m R_0(\tfrac{n}2 + \sigma, z,w)}  \le 
\begin{cases} 
C_{m} \abs{\log{r}} & \abs{r\sigma} \le 1, n=1, \\
C_{n,m}  r^{1-n} & \abs{r\sigma} \le 1, n\ge 2, \\
C_{n,m,\epsilon} \abs{\sigma}^{n/2-1} e^{-(\frac{n}2 + \re \sigma - \epsilon) r} & \abs{r\sigma} \ge 1, \end{cases}
\end{equation}
valid for $\re \sigma \ge 0$, $\abs{\sigma} \ge 1$, and $r \in (0,\infty)$.

For the imaginary part (on the critical line) the diagonal singularity is cancelled and we have the estimate, 
\begin{equation}\label{ImR0.kernel}
\abs{\del_\lambda^m \im R_0(\tfrac{n}2 + i\lambda, z,w)} \le C_n \abs{\lambda}^{n-1}, \qquad\text{for } \abs{r\lambda} \le 1
\end{equation}
valid for $\lambda \in \bbR$, $\abs{\lambda} \ge1$ and $r \in (0,\infty)$.
\end{corollary}

\begin{proof}
By the conjugation symmetry, it suffices to consider $\arg \sigma \in [0,\tfrac{\pi}2]$.

With $m=0$, the estimate \eqref{R0.kernel} follows from applying \eqref{QK.est}, \eqref{IK.zero}, and \eqref{IK.infty}
to \eqref{R0.legendre}.  The cases with $m>0$ essentially follow from analyticity and Cauchy's integral formula on a disk of radius $\epsilon$ centered at $\lambda$.  
For the cases with $\re \sigma =0$, this Cauchy estimate requires extending 
slightly beyond the range of Lemma~\ref{PQ.lemma}.  This is easily accomplished using 
the standard connection formula for the $Q$-Legendre function.  For the free resolvent kernel, the Legendre connection formula implies that
\[
\begin{split}
R_0(\tfrac{n}{2} +\sigma - 1; z,w) &= \frac{2\sigma}{\frac{n}2 + \sigma - 1} (\cosh r) R_0(\tfrac{n}{2} +\sigma;z,w) \\
&\qquad+ \frac{\sigma - \frac{n}2 + 1}{\frac{n}2 + \sigma - 1} R_0(\tfrac{n}{2} +\sigma + 1;z,w).
\end{split}
\]
This allows us to simply push the estimates for $\re \sigma \in [-\epsilon, \epsilon]$ to 
$\re \sigma \in [1-\epsilon,1+ \epsilon]$.

We establish \eqref{ImR0.kernel} in the same way, starting from \eqref{imR0} and using the P-Legendre asymptotics.
\end{proof}

The estimate on derivatives in Corollary~\ref{R0.cor} is not sharp, in the sense that 
the derivatives in \eqref{dR0.kernel} should cause a polynomial loss of decay in $r$, rather than exponential.
The extra level of precision would however be irrelevant for our application.

For $\abs{\sigma}\le 1$ the corresponding estimates can be derived much more directly from Legendre function asymptotics for fixed order.  
For later use, we note these in the following:
\begin{lemma}\label{R0.bottom}
Near $\sigma =0$ we have the bounds,
\begin{equation}\label{R0.kernelb}
\abs{R_0(\tfrac{n}2 + \sigma, z,w)}  \le \begin{cases} 
C \abs{\log r} & r \le 1, n=1, \\
C_{n}  r^{1-n} & r \le 1, n\ge 2, \\
C_{n} \abs{\sigma}^{n/2-1} e^{-(\frac{n}2 + \re \sigma) r} & r \ge 1, \end{cases}
\end{equation}
where $r :=d(z,w)$, valid for $\re \sigma \ge 0$, $\abs{\sigma} \le 1$, and $r \in (0,\infty)$.  
For any $\epsilon >0$,
\begin{equation}\label{dR0.kernelb}
\abs{\del_\sigma^m R_0(\tfrac{n}2 + \sigma, z,w)}  \le 
\begin{cases} 
C_{m} \abs{\log{d(z,w)}} & r \le 1, n=1, \\
C_{n,m}  d(z,w)^{1-n} & r \le 1, n\ge 2, \\
C_{n,m,\epsilon} \abs{\sigma}^{n/2-1} e^{-(\frac{n}2 + \re \sigma - \epsilon) r} & r \ge 1, \end{cases}
\end{equation}
valid for $\re \sigma \ge 0$, $\abs{\sigma} \le1$, and $r \in (0,\infty)$.

For the imaginary part we have 
\begin{equation}\label{ImR0.kernelb}
\abs{\del_\lambda^m \im R_0(\tfrac{n}2 + i\lambda, z,w)} \le C_n \abs{\lambda}^{n-1}, \qquad\text{for } \abs{r\lambda} \le 1
\end{equation}
valid for $\lambda \in \bbR$, $\abs{\lambda} \le 1$.
\end{lemma}

\section{Resolvent operator estimates}
\label{spectheory}

In this section we establish some weighted operator-norm estimates for the free resolvent and the perturbed resolvent in 
$\HS^{n+1}$.  As noted in the Introduction, the weights are expressed in terms of the function
\[
\rho := e^{-r},
\]
where $r$ is the radial coordinate in geodesic polar coordinates for $\bbH^{n+1}$. 

The first estimate is a slight extension of Guillarmou \cite[Prop.~3.2]{guillarmou2005absence}.  
We will include a proof for the convenience of the reader, but it follows the original proof fairly closely.  

\begin{proposition}\label{R0.bounds}
For the boundary defining function $\rho = e^{-r}$, and with $\eta>0$, 
and $\lambda \in \bbR$, we have
\[
\norm{\rho^{\eta} \del_\lambda^q R_0(\tfrac{n}2 + i\lambda) \rho^\eta}_{L^2 \to L^2} \le C_{q,\eta} \abs{\lambda}^{-1}.
\]
\end{proposition}

\begin{proof}
Define a family of radial cutoffs $\chi_t \in \cinf_0(\bbH^{n+1})$ such that, 
\[
\chi_t(z) = \begin{cases}0 & r \ge t/2, \\
1 & r \le t/4,
\end{cases}
\]
with $r := d(z,0)$.   

For the odd-dimensional case, we can use the cosine wave operator,
\[
U_0(t) :=  \cos\left(t \sqrt{-\Delta - \tfrac{n^2}{4}}\right).
\]
If $n+1$ is odd then the support of $U(t;z,w)$ is restricted to $\{d(z,w) = t\}$ by Huygen's principle, so that
\[
\chi_t U_0(t) \chi_t = 0,\quad\text{for }t>0.
\]
For $\eta>0$, we can use this to subdivide $\rho^\eta U_0(t) \rho^\eta$ as
\begin{equation}\label{rUr}
\begin{split}
\rho^\eta U_0(t) \rho^\eta &= (1-\chi_t) \rho^\eta U_0(t) (1-\chi_t) \rho^\eta + (1-\chi_t) \rho^\eta U_0(t) \chi_t \rho^\eta \\
&\qquad + \chi_t \rho^\eta U_0(t) (1-\chi_t) \rho^\eta.
\end{split}
\end{equation}
Note that for $z \in \supp (1-\chi_t)$, we have $r \ge t/4$, implying $\rho \le e^{-t/4}$.  This gives a bound
\[
\norm{(1-\chi_t) \rho^\eta}_\infty \le e^{-\eta t/4}.
\]
Since we also have $\norm{U_0(t)} \le 1$ and $\norm{\chi_t \rho^\eta}_\infty \le 1$, we deduce from \eqref{rUr} that
\[
\norm{\rho^\eta U_0(t) \rho^\eta} \le 3e^{-\eta t/4}.
\]
By the functional calculus,
\[
\rho^\eta R_0(\tfrac{n}2 + i\lambda) \rho^\eta = \frac{1}{i\lambda} \int_0^\infty e^{-it\lambda} \rho^\eta U_0(t) \rho^\eta\>dt.
\]
Hence we conclude for $\lambda \in \bbR$ that
\[
\norm{\rho^\eta R_0(\tfrac{n}2 + i\lambda) \rho^\eta} \le C\eta^{-1} \abs{\lambda}^{-1}.
\]
The same argument shows that
\[
\norm{\del_\lambda^q \rho^\eta R_0(\tfrac{n}2 + i\lambda) \rho^\eta} \le C_q\eta^{-(q+1)} \abs{\lambda}^{-1}.
\]

When the dimension is even, we start from the sine wave operator $U_1(t)$, 
related to the cosine operator by $\del_t U_1(t) = U_0(t)$.   For $n+1$ even the integral kernel is given by 
\begin{equation}\label{U1.even}
U_1(t,z,w) := C_n \left[ \sinh^2(t/2) - \sinh^2(d(z,w)/2) \right]_+^{-n/2}.  
\end{equation}
By writing 
\[
\chi_t U_0(t) \chi_t = \del_t(\chi_t U_1(t) \chi_t) - (\del_t\chi_t)U_1(t) \chi_t - \chi_t U_1(t) (\del_t\chi_t),
\]
we can conclude that $\chi_t U_0(t) \chi_t$ has smooth integral with support restricted to $d(z,w) \le t/2$.  
Using this restriction in conjunction with the formula \eqref{U1.even} gives
\begin{equation}\label{cUc}
\norm{\chi_t U_0(t) \chi_t} \le Ce^{-nt/2},
\end{equation}
for $t$ sufficiently large.  We now proceed as in the odd dimensional case.  
The expansion corresponding to \eqref{rUr} now has an extra term involving $\chi_t U_0(t) \chi_t$, which is controlled by 
\eqref{cUc}.  Assuming $\eta \le 2n$, we obtain the estimate 
\[
\norm{\rho^\eta U_0(t) \rho^\eta} \le Ce^{-\eta t/4},
\]
and the rest of the proof follows exactly as in the odd dimensional case.
\end{proof}

\bigbreak
For a potential $V \in \rho^\alpha L^\infty(\bbH^{n+1})$ with $\alpha > 0$, the operator norm $\norm{VR_0(s)}$ is small for 
$\re s$ large by the standard resolvent norm estimate on $R_0(s)$.  
Hence, the operator $1 + VR_0(s)$ is invertible by Neumann series for large $\re s$.  For $s$ in this range, 
the resolvent identity gives
\[
\begin{split}
R_V(s) &:=  (-\Delta +V - s(n-s))^{-1}\\
& = R_0(s)(1 + VR_0(s))^{-1}.
\end{split}
\]

Before discussing the estimates of the full resolvent on the critical line, we must first establish the meromorphic continuation that makes its extension to the critical line well-defined.
\begin{lemma}\label{mero.cont}
For $V \in \rho^\alpha L^\infty(\bbH^{n+1})$ with $\alpha > 0$, the resolvent $R_V(s)$ admits a meromorphic continuation 
to the half-plane $\re s > \tfrac{n}2 - \delta$ as a bounded operator 
\[
R_V(s): \rho^{\delta}L^2(\bbH^{n+1}) \to \rho^{-\delta}L^2(\bbH^{n+1}),
\]
for $\delta < \alpha/2$.
\end{lemma}
\begin{proof}
It follows from \cite[Prop.~3.29]{Mazzeo91} that $\rho^\alpha R_0(s)$ is compact as an operator on $\rho^\delta L^2(\bbH^{n+1})$
provided that $\re s > \tfrac{n}2 - \delta$ and $\alpha > 2\delta$.   This implies that
$VR_0(s)$ is compact on $\rho^\delta L^2(\bbH^{n+1})$ under the same conditions. 
Therefore, the analytic Fredholm theorem gives a meromorphic continuation of $R_V(s)$ to the half-plane $\re s > \tfrac{n}2 - \delta$.
\end{proof}
For this class of potentials, the high-frequency behavior of the resolvent on the critical line is unaffected by the potential.

\begin{proposition}
\label{prop:resest}
For $V \in \rho^{\alpha} L^\infty(\bbH^{n+1})$ with $\alpha>0$, there exists a constant $M_{V}$ such that for $\lambda \in \bbR$ with $\abs{\lambda} \ge M_{V}$,
\begin{equation}
\label{eqn:resest}
\norm{\rho^{\alpha/2} \del_\lambda^q R_V(\tfrac{n}2 + i\lambda)\rho^{\alpha/2}}_{L^2 \to L^2} \le C_{q,\alpha} \abs{\lambda}^{-1}.
\end{equation}
In particular, there are no resonances on the critical line for $\abs{\lambda} \ge M_V$.
\end{proposition}
\begin{proof} 
Similar results were proven in \cite{Borthwick2014} with slightly stronger assumptions on the potential.  By the resolvent identity,
\[
R_0(s) = R_V(s) + R_V(s)VR_0(s),
\]
we can write 
\[
R_0(s)\rho^{\alpha/2} = R_V(s)\rho^{\alpha/2}(1 + \rho^{-\alpha/2}VR_0(s)\rho^{\alpha/2}).
\]
The factor on the right is meromorphically invertible by the analytic Fredholm theorem, so that
\begin{equation}\label{Rvp}
R_V(s)\rho^{\alpha/2} = R_0(s)\rho^{\alpha/2} (1 + \rho^{-\alpha/2}VR_0(s)\rho^{\alpha/2})^{-1}.
\end{equation}
By Proposition \ref{R0.bounds},
\[
\norm{\rho^{-\alpha/2}VR_0(\tfrac{n}2 + i\lambda)\rho^{\alpha/2}} \le C_\alpha \norm{\rho^{-\alpha} V}_\infty \abs{\lambda}^{-1}.
\]
Hence for $V \in \rho^{\alpha} L^\infty(\bbH^{n+1})$, there exists a constant $M_{V}$ such that for 
$\abs{\lambda} \ge M_{V}$,
\[
\norm{\rho^{-\alpha/2}VR_0(\tfrac{n}2 + i\lambda)\rho^{\alpha/2}} \le \frac12,
\]
implying that $(1 + \rho^{-\alpha/2}VR_0(\tfrac{n}2 + i\lambda)\rho^{\alpha/2})^{-1}$ exists and satisfies
\[
\norm{(1 + \rho^{-\alpha/2}VR_0(\tfrac{n}2 + i\lambda)\rho^{\alpha/2})^{-1}} \le 2.  
\]
The estimates then follow from \eqref{Rvp} and Proposition \ref{R0.bounds}.  
\end{proof}

For the matrix case, we also need corresponding estimates for $R_V(\tfrac{n}2 + \sigma)$ with $\sigma>0$, but these estimates just follow from the standard formula for the resolvent norm in terms of distance to the spectrum, with no need for weights.  
For $\sigma$ sufficiently large, we have
\[
\norm{R_V(\tfrac{n}2 + \sigma)}_{L^2 \to L^2}  = O(\sigma^{-2}).  
\]
By writing $\sigma$-derivatives in terms of powers of the resolvent, we can extend this to
\begin{equation}\label{RV.realaxis}
\norm{\del_\sigma^q R_V(\tfrac{n}2 + \sigma)}_{L^2 \to L^2}  = O(\sigma^{-2-q}),
\end{equation}
for $q = 0, 1, 2,\dots$ and $\sigma>0$.

\section{Absence of Embedded Resonances and Eigenvalues}\label{absence} 

In this section we take up the proof of Theorem~\ref{absence.thm}.
Proposition~\ref{prop:resest} already established the absence of embedded eigenvalues and resonances for $\lambda$ sufficiently large by showing that $R_V(\tfrac{n}2 + i\lambda)$ is regular for large $\abs{\lambda}$.

To extend this result to all $\lambda \ne 0$, our first task is to show that any resonances on the critical line must 
come from an embedded eigenvalue, except possibly at the bottom of the spectrum.
We will subsequently show that such embedded eigenvalues are ruled out.  Similar results were established in \cite{lawrie2014stability} for the Schr\"odinger operator associated with the wave maps problem on $\HS^{n+1}$.

\begin{lemma}\label{res.crit.lemma}
Suppose $V \in \rho^\alpha L^\infty(\HS^{n+1}, \bbR)$ for some $\alpha>0$.  If $R_V(s)$ has a pole at 
$s = \tfrac{n}{2} + i\lambda$ for $\lambda \in \bbR\backslash\{0\}$, 
then $\tfrac{n^2}{4} + \lambda^2$ is an embedded eigenvalue for $-\Delta + V$.
\end{lemma}
\begin{proof}
If $R_V(s)$ has a pole at $s = \tfrac{n}{2} + i\lambda$, then for $\phi \in \cinf_0(\HS^{n+1})$ we have
\[
R_V(s) \phi = (s - \tfrac{n}2 - i\lambda)^{-m} u +  (s - \tfrac{n}2 - i\lambda)^{-m+1}v(s),
\]
for some $m \ge 1$, with $v(s)$ analytic near $s = \tfrac{n}{2} + i\lambda$.  Applying $-\Delta + V - s(n-s)$ gives 
\[
(-\Delta + V - s(n-s))u = (s - \tfrac{n}2 - i\lambda)^{m} \phi - (s - \tfrac{n}2 - i\lambda) (-\Delta + V - s(n-s))v(s).
\]
Taking $s \to \tfrac{n}2 + i\lambda$ then shows that
\begin{equation}\label{u.eigen}
(-\Delta + V - \tfrac{n^2}4 - \lambda^2)u = 0.
\end{equation}
By the identity $R_V(s) = R_0(s) - R_0(s) VR_V(s)$, we see that $R_V(s)$ maps $\cinf_0(\HS^{n+1}) \to \rho^{-\epsilon} H^2(\HS^{n+1})$ for any $\epsilon>0$.  Hence $u \in \rho^{-\epsilon} H^2(\HS^{n+1})$.  

It remains to prove that $u$ actually lies in $L^2(\HS^{n+1})$.  If we set $\epsilon = \alpha/2$ and then the assumption on $V$ gives
\[
(-\Delta  - \tfrac{n^2}4 - \lambda^2)u \in \rho^{\alpha/2} L^2(\HS^{n+1}).
\]
We can apply \cite[Thm.~7.14]{Mazzeo91} to deduce that 
\begin{equation}\label{u.expand}
u = \rho^{\frac{n}2 + i\lambda} a + \rho^{\alpha/2} v,
\end{equation}
where $a$ is a function on the sphere $S^n$ and $v \in H^2(\bbH^{n+1})$.   
We could also have deduced this directly from $u = -R_0(\tfrac{n}2 + i\lambda) Vu$ and
the explicit formula for the kernel of $R_0(s)$.
From the fact that $u \in \rho^{-\alpha/2} L^2(\HS^{n+1})$ we can deduce that $a \in L^2(S^n)$.

Note that $u \in L^2(\HS^{n+1})$ if and only if $a = 0$.  To prove that $a=0$ we will use a boundary pairing argument adapted from \cite{Melrose94}.   
Let $\psi \in \cinf(\bbR)$ be a function with $\psi(r) = 0$ for $r \le 1$, $\psi(r) = 1$ for $r \ge 2$, and $\psi'(r) \ge 0$.
Then for $\delta>0$ let $\psi_\delta(r) := \psi(e^{-r}/\delta)$.  We compute the commutator,
\[
\begin{split}
[-\Delta, \psi_\delta] &= -\delta^{-2} e^{-2r} \psi''(e^{-r}/\delta) + \delta^{-1} (n\coth^2 r - 1) e^{-r}  \psi'(e^{-r}/\delta)\\
&\qquad - 2\delta^{-1} e^{-r}  \psi'(e^{-r}/\delta)\del_r.
\end{split}
\]
By the eigenvalue equation \eqref{u.eigen} and the fact that $u \in \rho^{-\epsilon} H^2(\HS^{n+1})$, we have
\begin{equation}\label{delta.psi}
\brak{[-\Delta, \psi_\delta]u,u} = 0.
\end{equation} 
(This is the point where we must assume that the potential $V$ is real.)  
If we substitute \eqref{u.expand} into this inner product then
since $\rho \le 2\delta$ on the support of $[-\Delta, \psi_\delta]$, the contribution from the $\rho^{\alpha/2} v$ terms will be 
$O(\delta^{\alpha/2})$ as $\delta \to 0$.

Thus,
\begin{equation}\label{rho.a.limit}
\lim_{\delta\to 0} \brak{[-\Delta, \psi_\delta] \rho^{\frac{n}2 + i\lambda} a,\rho^{\frac{n}2 + i\lambda} a} = 0.
\end{equation}
Evaluating the left-hand side gives
\[
\begin{split}
&\lim_{\delta \to 0} \int_{-\log 2\delta}^{-\log\delta} \Bigl(-\delta^2 e^{(n-2)r} \psi''(e^{-r}/\delta)
+ \delta^{-1} (n\coth^2 r - 1) e^{(n-1)r} \psi'(e^{-r}/\delta) \\
&\hskip1in - (n+2i\lambda) \delta^{-1} e^{(n-1)r} \psi'(e^{-r}/\delta) \Bigr) \sinh^{n} r\>dr \cdot  \norm{a}^2_{L^2(S^n)}\\
&\qquad= 2i\lambda \norm{a}^2_{L^2(S^n)}.
\end{split}
\]
We conclude from \eqref{rho.a.limit} that for $\lambda \ne 0$ we must have $a = 0$, implying that $u$ is an honest $L^2$-eigenfunction.
\end{proof}

\begin{remark}
For the non-selfadjoint matrix equation, in place of \eqref{delta.psi} we would have
\[
\brak{[-\Delta, \psi_\delta] \vec{u}, \vec{u}} = 2 i \psi_\delta \im (u^2),
\] 
so this technique could not be used rule out resonances in that case.
\end{remark}

\bigbreak
Having shown that resonances on the critical line must come from embedded eigenvalues, 
our next step is to rule out the embedded eigenvalues.  This can be done under a weaker decay assumption.

\begin{proposition}\label{embed.prop}
Suppose $V \in  L^\infty(\HS^{n+1})$, with $V = o(r^{-1})$ as $r \to \infty$.  
Then $-\Delta + V$ has no eigenvalues in the range $(\tfrac{n^2}{4}, \infty)$.
\end{proposition}
\begin{proof}
The argument from \cite{Donnelly81} essentially carries over directly to the Schr\"odinger case, even for non-smooth potentials.
Suppose $u \in H^2(\HS^{n+1})$ satisfies the eigenvalue equation,
\[
(-\Delta + V - \tfrac{n^2}{4} - \lambda^2)u = 0.
\]
If we write $w = (\sinh r)^{n/2} u$ then the equation becomes
\[
Hw = \lambda^2 w,
\]
where
\[
H := -\del_r^2 - (\sinh r)^{-2} \Delta_{S^n} + \tilde{V},
\]
where
\[
\tilde{V} := V + \frac{n(n-2)}{4}(\coth^2 r - 1).
\]

Since $u \in H^2(\HS^{n+1})$, we have $w \in H^2(\bbR_+ \times S^n)$ and thus 
the trace $w|_r$ is well-defined in $H^{3/2}(S^n)$ for each $r$, allowing us to define
\[
G(r) := \norm{\del_r w}^2_{L^2(S^n)} - (\sinh r)^{-2} \norm{\nabla_\theta w}^2_{L^2(S^n)} +  \lambda^2 \norm{w}^2_{L^2(S^n)}.
\]
The fact that $w \in H^2(\bbR_+ \times S^n)$ further implies that $G \in L^1(\bbR_+)$.
Using the eigenvalue equation, we can calculate that
\[
\begin{split}
(rG(r))' &= \norm{\del_r w}^2_{L^2(S^n)} + \lambda^2 \norm{w}^2_{L^2(S^n)} - \left(\frac{r}{\sinh^2 r}\right)' \norm{\nabla_\theta w}^2_{L^2(S^n)} \\
&\qquad+ 2r (\del_r w, \tilde{V}w)_{L^2(S^n)}.
\end{split}
\]
The first three terms are positive, and the fourth is bounded by the first two for $r$ sufficiently large, by the Cauchy-Schwarz inequality and the assumption that $V = o(r^{-1})$.  We conclude that $(rG(r))' \ge 0$ for all 
$r \ge R_0$ with $R_0$ sufficiently large.
The integrability of $G$ would evidently fail if we had $G(r) >0$ for any $r \ge R_0$, 
so we can conclude that $G(r) \le 0$ for $r \ge R_0$.

The remainder of the proof follows \cite{Donnelly81} very closely.  We set $w_m := r^{m} w$ and
\[
\begin{split}
L_m(r) &:=  \norm{\del_r w_m}^2_{L^2(S^n)} - (\sinh r)^{-2} \norm{\nabla_\theta w_m}^2_{L^2(S^n)} \\
&+ \Bigl[\lambda^2(1 - R_0r^{-1}) + m(m+1) r^{-2} \Bigr] \norm{w}^2_{L^2(S^n)}.
\end{split}
\]
A computation similar to that for $G(r)$,  using the assumption that $V = o(r^{-1})$, shows that $(r^2 L_m(r))' >0$ for $m > m_0$ and $r> R_1>R_0$.  It follows that $L_{m_1}(r)>0$ for some $m_1>m_0$ and $r>R_2>R_1$.  We can then chose 
$R_3>R_2$ such that $\del_r\norm{w}_{L^2(S^n)}|_{r=R_3} < 0$ and so that  
\[
-\lambda^2 R_0r^{-1} + m_1(2m_1+1)r^{-2} < 0,
\]
for $r \ge R_3$.  A direct estimate then shows that
\[
R_3^{-2m_1} L_{m_1}(R_3) \le G(R_3).
\]
Since $L_{m_1}(R_3)>0$ and $G(R_3)<0$, this contradiction rules out the existence of an eigenvector.
\end{proof}

The combination of Lemmas~\ref{mero.cont} and \ref{res.crit.lemma} with Proposition~\ref{embed.prop} furnishes the 
proof of Theorem~\ref{absence.thm}.

\section{Full spectral resolution estimates}
\label{sec:specres}

Although it is relatively straightforward to produce operator norm bounds on the full resolvent $R_V(s)$, 
the dispersive estimates require finer control of the kernel of the spectral resolution,
\[
d\Pi_V(\lambda) = -4\lambda \im [R_V(\tfrac{n}2 - i\lambda)] \>d\lambda.
\]  
Translating the operator bounds on $R_V(s)$ into pointwise bounds on the imaginary part of the
kernel is the main goal of this section.  

\begin{proposition}\label{ImRV.prop}
Assume that $V \in \rho^\alpha L^\infty$, where 
\[
\frac{\alpha}n > 1 - \left\lfloor \frac{n+5}4 \right\rfloor^{-1},
\]
and assume that $-\Delta + V$ does not have a resonance at $s = \tfrac{n}2$.  
Then there exists an $M>0$, such that for any $q \in \mathbb{N}$,
\[
\sup_{z,w\in \HS^{n+1}} \abs{\del_\lambda^q \im R_V(\tfrac{n}2 + i\lambda; z, w)} \le C_{q,V} \brak{\lambda}^M,
\]
for all $\lambda \in \bbR$.
\end{proposition}

The restriction on $\alpha$ is trivially satisfied by any $\alpha >0$ for $n=1,2$.  
For $3 \le n \le 6$ the condition is $\alpha>n/2$, which means that $V$ must have decay just slightly better than $L^2$.  
For $n>6$ the required decay is intermediate between $L^2$ and $L^1$.

The strategy for the proof of Proposition \ref{ImRV.prop} is to combine weighted $L^p$ estimates on the kernels 
using an analog of Young's inequality.  Let us first establish the kernel estimates.

\begin{lemma}\label{R0.Lp.lemma}
For $\lambda\in\bbR$, we have
\begin{equation}\label{R0.qnorm}
\norm{\del_\lambda^m R_0(\tfrac{n}2 + i\lambda; z,\cdot) \rho^\alpha}_{L^q} \le C_{n,m,q,\alpha} \brak{\lambda}^{n-1},
\qquad \text{for }1 \le q < \frac{n+1}{n-1},
\end{equation}
and
\begin{equation}\label{Im.qnorm}
\norm{\del_\lambda^m \Im R_0(\tfrac{n}2 + i\lambda; z,\cdot) \rho^\alpha}_{L^q} \le C_{n,m,q,\alpha} \brak{\lambda}^{n-1},
\qquad \text{for }1 \le q \le \infty,
\end{equation}
provided that $\alpha > \max\{0, n(\tfrac1{q} - \tfrac12)\}$.  
The estimates are uniform for $z \in\HS^{n+1}$.
\end{lemma}
\begin{proof}
For $\abs{\lambda} \ge 1$, the idea is to split the $q$-norm, 
\[
\norm{\del_\lambda^m R_0(\tfrac{n}2 + i\lambda; z,\cdot) \rho^\alpha}^q_{L^q} 
= \int \abs{\del_\lambda^mR_0(\tfrac{n}2 + i\lambda; z,w)}^q \rho(w)^{\alpha q} \> d \Vol_w.
\]
according to the value of $r\lambda$, where $r:= d(z,w)$, and use the estimates of Corollary~\ref{R0.cor}.  

For $\lambda d(z,w) \le 1$ we can drop the $\rho$ factor (since $\| \rho \|_{L^\infty} \leq C$) and use \eqref{R0.kernel} and \eqref{dR0.kernel} to write
\[
\begin{split}
\int_{\lambda d(z,w)\le 1} \abs{\del_\lambda^m R_0(\tfrac{n}2 + i\lambda; z,\cdot)}^q \rho^{\alpha q} \> d \Vol_w
& \le C_{n,m} \int_{\lambda r \le 1} r^{(1-n)q} \>\sinh^n r\>dr \\
& \le C_{n,m,q} \lambda^{(n-1)q - n-1},
\end{split}
\]
assuming that $(n-1)q < n+1$.  For $\lambda d(z,w) \ge 1$ we have
\[
\begin{split}
&\int_{\lambda d(z,w)\ge 1} \abs{\del_\lambda^m R_0(\tfrac{n}2 + i\lambda; z,\cdot)}^q \rho^{\alpha q} \> d \Vol_w \\
&\qquad \le C_{n,m,\epsilon} \lambda^{(n/2-1)q} \int e^{-q(\frac{n}2-\epsilon) d(z,w)} e^{-\alpha q d(w,0)} \> d \Vol_w 
\end{split}
\]
To eliminate the $z$ dependence, we split the terms with H\"older's inequality,
\[
\int e^{-q(\frac{n}2-\epsilon) d(z,w)} e^{-\alpha q d(w,0)} \> d \Vol_w \le \norm{e^{-q(\frac{n}2-\epsilon)r}}_p \norm{e^{-\alpha qr}}_{p'},
\]
where $p,p'$ are conjugate.  Since the measure includes a weight $\sinh^n r$, by choosing $\epsilon$ sufficiently small we can make these norms finite provided that $qp>2$ and $\alpha q p' >n$.  
Such a choice of $p,p'$ is possible provided that $\alpha > n(1/q - 1/2)$ and $\alpha \ge 0$.  
Hence, under these conditions,
\[
\int_{\lambda d(z,w)\ge 1} \abs{\del_\lambda^m R_0(\tfrac{n}2 + i\lambda; z,\cdot)}^q \rho^{\alpha q} \> d \Vol_w
\le C_{n,m,q,\alpha} \lambda^{(n/2-1)q}.
\]
This completes the proof of \eqref{R0.qnorm}.

For \eqref{Im.qnorm} the argument is essentially identical, except that we use \eqref{ImR0.kernel} to improve the estimate for $\lambda d(z,w) \le 1$.

For $\abs{\lambda}\le 1$, we use Lemma~\ref{R0.bottom} to estimate the kernels, and the integrals are split into $r \le 1$ and $r \ge 1$.  Otherwise the estimates proceed just as above.
\end{proof}

\bigbreak
To apply the $L^q$ estimates, we need a version of Young's inequality.   
Since we are not actually dealing with convolutions, we need to be a bit careful about the estimates required for the kernels.
\begin{lemma}\label{young.lemma}
On a measure space $(X,\mu)$, suppose the integral kernels $K_j(z,w)$ satisfy uniform estimates
\[
\norm{K_1(\cdot, w)}_{L^{q_1}} \le A_{q_1}, \qquad
\norm{K_1(z,\cdot)}_{L^{q_1}} \le A_{q_1}, \qquad
\norm{K_2(\cdot, z')}_{L^{q_2}} \le B_{q_2},
\]
for $q_1,q_2,p \in [1,\infty]$ such that
\[
\frac{1}{q_1} + \frac{1}{q_2} = \frac{1}{p} + 1.
\]
Then we have
\[
\norm{\int K_1(\cdot, w) K_2(w, z') \> d\mu(w)}_{L^{p}} 
\le A_{q_1} B_{q_2},
\]
uniformly in $z'$.  (The bound on $\norm{K_1(\cdot, w)}_{L^{q_1}}$ is not required if $p=\infty$.)
\end{lemma}
\begin{proof}
For $p=\infty$, the result follows immediately by H\"older, so we can assume $p<\infty$, which implies $q_1,q_2 < \infty$ also.

Set 
\[
h(z,z') := \int K_1(z, w) K_2(w,z') \> d\mu(w).
\]
If we set $s = q_2(1-1/q_1) \in [0,1]$,
we can split $$\abs{K_1(z, w)} = \abs{K_1(z, w)}^s \abs{K_1(z, w)}^{1-s}$$ and then apply H\"older to obtain
\begin{align*}
&  \abs{h(z,z')} \le \\
& \left[ \int \abs{K_1(z, w)}^{sq_1'} \>d\mu(w) \right]^{\frac{1}{q_1'}}  
 \left[ \int \abs{K_1(z, w)}^{(1-s)q_1} \abs{K_2(w, z')}^{q_1} d\mu(w) \right]^{\frac1{q_1}},
\end{align*}
where $1/q_1' = 1 - 1/q_1 = s/q_2$.  This then implies that
\[
\abs{h(z,z')}^{q_1} \le A_{q_2}^{sq_1} \int \abs{K_1(z, w)}^{(1-s)q_1} \abs{K_2(w, z')}^{q_1} d\mu(w).  
\]  
Now we take the $p/q_1$-norm with respect to $z$ on both sides, yielding
\[
\norm{h(\cdot,z')}_{L^p}^{q_1} \le A_{q_2}^{sq_1} 
\left[\int \left(\int \abs{K_1(z, w)}^{(1-s)q_1} \abs{K_2(w, z')}^{q_1} d\mu(w) \right)^{\frac{p}{q_1}} d\mu(z) \right]^{\frac{q_1}{p}}.  
\]
We can use the Minkowski integral inequality to switch the order of integration and then apply the assumed $L^p$ bounds
(noting that $(1-s)p = q_2$):
\[
\begin{split}
& \left[\int \left(\int \abs{K_1(z, w)}^{(1-s)q_1} \abs{K_2(w, z')}^{q_1} d\mu(w) \right)^{\frac{p}{q_1}} d\mu(z) \right]^{\frac{q_1}{p}} \\
& \qquad \le  \int \left(\int \abs{K_1(z, w)}^{(1-s)p} \abs{K_2(w, z')}^{p} d\mu(z) \right)^{\frac{q_1}{p}} d\mu(w) \\
& \qquad \le  \int A_p^{(1-s)q_1} \abs{K_2(w, z')}^{p} d\mu(w) \\
&  \qquad \le A_{q_2}^{(1-s)q_1} B_{q_1}^{q_1}.
\end{split}
\]
Combining these estimates gives
\[
\norm{h(\cdot,z')}_{L^p}^{q_1} \le A_{q_2}^{q_1} B_{q_1}^{q_1},
\]
which completes the proof.
\end{proof}

\bigskip
\begin{proof}[Proof of Proposition~\ref{ImRV.prop}]
As in the recent results \cite{GolGreen1,GolGreen2}, the proof relies on the Birman-Schwinger type resolvent expansion at all frequencies:
\begin{equation}\label{eqn:RE-BS}
\begin{split}
R_V(s) &= \sum_{\ell = 0}^{2m-1} R_0(s) \bigl[- V R_0(s)\bigr]^\ell \\
&  \hspace{.6cm}  + \bigl[ R_0(s) V\bigr]^{m} R_V(s) \bigl[V R_0(s)\bigr]^{m}.
\end{split}
\end{equation}
In the free resolvent kernel estimates of Lemma~\ref{R0.Lp.lemma}, the derivatives with respect to $\lambda$ do not affect the order of growth in $\lambda$.   The same holds true of the full resolvent operator estimates in Proposition~\ref{prop:resest}.
We may thus focus on the undifferentiated case, as taking derivatives will merely change the constants.

Let us first focus on the remainder term in the series expansion \eqref{eqn:RE-BS}, 
\[
\bigl[ R_0(\tfrac{n}2 + i\lambda) V\bigr]^{m} R_V(\tfrac{n}2 + i\lambda) 
\bigl[V R_0(\tfrac{n}2 + i\lambda)\bigr]^{m},
\]
since the behavior of this term drives the choice of $m$.  We do not include the imaginary part here because that 
does not provide any advantage for the remainder term.  
Using the assumption that $V \in \rho^\alpha L^\infty$, we can write the kernel of this operator as an $L^2$-pairing,
\[
\bigl[ R_0(\tfrac{n}2 + i\lambda) V\bigr]^{m} R_V(\tfrac{n}2 + i\lambda) \bigl[V R_0(\tfrac{n}2 + i\lambda)\bigr]^{m}(z,z')
= \left\langle \rho^{\alpha/2} R_V \rho^{\alpha/2} 
h_{z'}, h_z \right\rangle_{L^2},
\]
where 
\[
h_z := \bigl[ \rho^{\alpha/2} R_0(\tfrac{n}2 - i\lambda) \rho^{\alpha/2}\bigr]^{m}\rho^{-\alpha/2}(z,\cdot).
\]
With the hypothesis that $R_V(\tfrac{n}2 + i\lambda)$ has no pole at $\lambda = 0$, we can extend the estimate of 
Proposition~\ref{prop:resest} through $\lambda = 0$ to give
\begin{equation}\label{BSm.bound}
\begin{split}
&\Bigl|\bigl[ R_0(\tfrac{n}2 + i\lambda) V\bigr]^{m} R_V(\tfrac{n}2 + i\lambda) \bigl[V R_0(\tfrac{n}2 + i\lambda)\bigr]^{m}(z,z')\Bigr| \\
&\qquad\le C_\alpha \brak{\lambda}^{-1} \norm{h_z}_{L^2} \norm{h_{z'}}_{L^2},
\end{split}
\end{equation}

Applying Lemma~\ref{young.lemma} iteratively gives the estimate
\begin{equation}\label{Aqm.bound}
\norm{h_z}_{L^2}
\le \sup_z \norm{ \rho^{\alpha/2} R_0(\tfrac{n}2 + i\lambda) \rho^{\alpha/2}(\cdot, z)}_{L^q}^m,
\end{equation}
provided that
\[
q = \frac{2m}{2m-1}.
\]
The left- and right-sided estimates needed for Lemma~\ref{young.lemma} are identical by the symmetry of $R_0(s;z,w)$.  Note that we do not have a weight factor $\rho^{\alpha/2}$ on the right for the final $R_0$ term in $h_z$, 
but fortunately Lemma~\ref{young.lemma} shows that we only need the estimate of this term in the left variable.  

Lemma~\ref{R0.Lp.lemma} applies to the right-hand side of \eqref{Aqm.bound} to give the estimate 
\begin{equation}\label{Aqm.bound2}
\norm{h_z}_{L^2} \le C_{n,q,\alpha} \brak{\lambda}^{n-1},
\end{equation}
provided $1 \le q < \frac{n+1}{n-1}$ and $\alpha/2 > n(1/q - 1/2)$.  The first requirement translates to 
\[
m > \frac{n+1}4,
\]
while the second requires that
\[
\alpha > n \left(\frac{m-1}{m}\right).
\]
Under these conditions, the combination of \eqref{BSm.bound} and \eqref{Aqm.bound2} gives
\[
\Bigl| \bigl[ R_0(\tfrac{n}2 + i\lambda) V\bigr]^{m} R_V(\tfrac{n}2 + i\lambda) 
\bigl[V R_0(\tfrac{n}2 + i\lambda)\bigr]^{m}(z,z') \Bigr|
\le C_{n,m,\alpha} \brak{\lambda}^{2m(n-1)-1},
\]
uniformly in $z,z'$.

Now that we have the condition on $m$, let us consider the imaginary part of 
a typical term in the expansion \eqref{eqn:RE-BS},
\[
\im \left(R_0(\tfrac{n}2 + i\lambda) [V R_0(\tfrac{n}2 + i\lambda)]^\ell \right),
\]
for $l = 0, \dots, 2m-1$.  Here taking the imaginary part is actually crucial.  We can expand the product so that each term has $\im R_0$ appearing as a factor in some position.  This guarantees that we can apply the estimate \eqref{Im.qnorm} to one of the $R_0$ factors.  For the factors of $R_0$ without imaginary part we are restricted to $L^q$ estimates with $q < \tfrac{n+1}{n-1}$, but for the $\im R_0$ term we can take any $1 \le q' \le \infty$.
For the estimates it does not make a difference which factor carries the imaginary part, so we can treat all the terms by the same approach.

By successive applications of Lemma~\ref{young.lemma}, using again the bounds from Lemma~\ref{R0.Lp.lemma}, we have
\[
\Bigl| \im \left(R_0(\tfrac{n}2 + i\lambda) [V R_0(\tfrac{n}2 + i\lambda)]^\ell \right) (z,z') \Bigr|
\le C \brak{\lambda}^{\ell(n-1)},
\]
provided that
\[
\frac{\ell}{q} + \frac{1}{q'} = \ell,
\]
and
\[
\frac{\alpha}2 > n \left(\frac{1}{q} - \frac12 \right), \qquad \frac{\alpha}2 > n\left(\frac{1}{q'} - \frac{1}{2}\right).
\]
If we take $q$ just below $\frac{n+1}{n-1}$, then $q'$ lies just above $\frac{n+1}{2\ell}$.  With such choices it's not hard to check that the conditions on $\alpha$ can be satisfied if $\alpha > n(n-3)/(n+1)$ and $\alpha > (2\ell - 2)/(n+1)$.  For any $n$ these requirements are weaker than the condition 
$\alpha > n(m-1)/m$ coming from the remainder term.
\end{proof}

\begin{remark}
Above, we have used the resolvent to approach analysis of the spectral measure.  Another approach could be to construct a modified Eisenstein series solution similar to the analysis in the work \cite{dyatlov2012microlocal}, which would parallel the distorted Fourier basis approach of \cite{Mar-lin}.  Such a study could be of independent interest and perhaps lead to a better understanding of the sharpness of our decay assumptions on $V$.  
\end{remark}

\section{Dispersive estimates: scalar case}
\label{sec:scalarinhom}

In this section, we proceed to prove Theorem \ref{thm1}.  With the hyperbolic convention for the spectral parameter, Stone's formula 
gives the continuous component of the spectral resolution as 
\begin{equation}
\label{F.def} 
\begin{split}
d\Pi_V (\lambda) &:= 2i\lambda \left[R_V (\tfrac{n}2 + i\lambda) - R_V (\tfrac{n}2 - i\lambda)\right] d\lambda \\
&=  -4\lambda \left[\im R_V (\tfrac{n}2 + i\lambda)\right] d\lambda.
\end{split} 
\end{equation}
We can then write the kernel of the Schr\"odinger propagator as
\begin{equation}
\label{eqn:specrep_inhom}
\left[e^{i t(-\Delta + V)}P_c\right] (z,w) = \frac{1}{i \pi} \int_0^\infty \int_{\HS^{n+1}}  e^{i t \lambda^2} \> d\Pi_V(\lambda; z,w).
\end{equation}

For both the high and low frequencies the dispersive estimates can now be derived from a combination of \eqref{ImRV.prop} and
integration by parts.

\begin{proof}[Proof of Theorem~\ref{thm1}]

Let $\chi \in \cinf_0(\bbR^+)$ be a cutoff function with $\chi(\lambda) = 1$ for $\lambda \le 1$.  We first consider the $t$ dependence of the high-frequency term,
\begin{equation}
\label{high-freq}
\int_0^\infty \int_{\HS^{n+1}} (1 - \chi(\lambda)) 
e^{i t \lambda^2} g(w) \>d\Pi_V(\lambda; z,w) \> d \Vol_w .
\end{equation}
Assuming that $V$ satisfies the hypotheses, we claim that for any $N>0$ and $R>1$, we have
\[
\sup_{z,w \in \HS^{n+1}} \abs{\int_0^\infty \chi(\lambda/R) (1 - \chi(\lambda)) 
e^{i t \lambda^2} \>d\Pi_V(\lambda; z,w)} \le C_{N,V} t^{-N},
\]
for $t>1$, where $C_{N,V}$ is independent of $R$.

Writing the spectral resolution as in \eqref{F.def} and integrating by parts $N$ times gives
\[
\begin{split}
&\int_0^\infty  \chi(\lambda/R) (1 - \chi(\lambda)) e^{i t \lambda^2} \>d\Pi_V(\lambda; z,w) = \\
&\ \
C_N t^{-N} \int_1^\infty e^{i t \lambda^2} \left(\lambda^{-1}\del_\lambda\right)^N 
\Bigl[ \chi(\lambda/R) (1 - \chi(\lambda)) \lambda \im R_V(\tfrac{n}2 + i\lambda; z,w)\Bigr] \>d\lambda.
\end{split}
\]
By Proposition~\ref{ImRV.prop}, we have 
\[
\sup_{z,w} \abs{\left(\lambda^{-1}\del_\lambda\right)^N 
[\chi(\lambda/R) (1 - \chi(\lambda) ] \lambda \im R_V(\tfrac{n}2 + i\lambda; z,w)} \le C_{N,V} \brak{\lambda}^{M-N+1},
\]
uniformly in $z,w$, and the result follows by direct $L^1$ estimate provided we take $N>M+1$.

Now let us consider the low-frequency term.  Note that since $d\Pi(\lambda)$ is an even function of $\lambda$, 
a single integration by parts gives
\[
\int_0^\infty  \chi(\lambda) e^{i t \lambda^2} \>d\Pi_V(\lambda; z,w) 
= Ct^{-1} \int_{-\infty}^\infty e^{i t \lambda^2} f'(\lambda)\>d\lambda,
\]
where
\[
f(\lambda) := \chi(\abs{\lambda}) \im[R_V(\tfrac{n}2 + i\lambda;z,w)],
\]
and we have exploited the conjugation symmetry to extend the integral to $\bbR$.   
The dispersive bound for the free one-dimensional Schr\"odinger equation now gives
\[
\abs{\int_0^\infty  \chi(\lambda) e^{i t \lambda^2} \>d\Pi_V(\lambda; z,w)} 
\le Ct^{-3/2} \norm{\mathcal{F}^{-1}(f')}_{L^1(\bbR)},
\]
for $t \geq 1$ 
With the simple bound,
\[
\norm{\mathcal{F}^{-1}(f')}_{L^1(\bbR)} \le C(\norm{f'}_{L^1 (\bbR)} + \norm{f'''}_{L^1 (\bbR)}),
\]
the low-frequency result then follows from the estimates in Proposition~\ref{ImRV.prop}.
\end{proof}

\section{Dispersive Estimates: matrix case}
\label{sec:matrix}

In this section, we construct the matrix Schr\"odinger operator spectral resolution and proceed to prove Theorem \ref{thm2}.  
Using the strategy from the scalar case, analysis in the matrix case boils down to resolvent estimates for a free Hamiltonian of the form
\begin{equation}
\label{eqn:H0}
\calH_0 := \begin{pmatrix} -\Delta + (\mu - \tfrac{n^2}4) & 0 \\
0 & \Delta - (\mu - \tfrac{n^2}4)  \end{pmatrix}.
\end{equation}
The spectrum is clearly $\sigma(\calH_0) = (-\infty, -\mu] \cup [\mu,\infty)$.
For $z \in \bbC - \sigma(\calH_0)$, the resolvent of $\calH_0$ is related to the free scalar resolvent by
\[
(\calH_0 - z)^{-1} = \begin{pmatrix}
R_0\left(\frac{n}{2} + \sqrt{\mu - z} \right)  &  0 \\
0 & - R_0\left(\frac{n}{2} +  \sqrt{\mu + z} \right)
\end{pmatrix},
\]
with the principal branch of the square root used for $\sqrt{\mu \pm z}$.  

We easily observe via estimates from Section \ref{spectheory} that the operator
\begin{equation*}
(\calH_0 - z)^{-1} V
\end{equation*}
is compact as an operator on $L^2$, 
for $V$ a matrix potential operator with components in $\rho^\alpha L^{\infty}(\bbH^{n+1}, \bbR)$ with $\alpha>0$
and $z \in \bbC - \sigma(\calH_0)$.
This allows us to define the perturbed matrix resolvent as in \cite[Lemma 4]{ES1}, by applying the Fredholm alternative
to the formula
\begin{equation*}
(\calH - z)^{-1} = (I + ( \calH_0 - z)^{-1} V)^{-1} ( \calH_0 - z)^{-1}.
\end{equation*}
As noted in the introduction, we can see from this that the continuous spectrum of $\calH$ is 
$\sigma(\calH_0)$ and that otherwise the spectrum of $\calH$ is purely discrete.  

In the free case, the estimates of Section \ref{spectheory} also imply a limiting absorption principle extending the resolvent to the continuous spectrum as an operator on the weighted space $\rho^\delta L^2$.  
If we set $m = \mu + \tau^2$ with $\tau >0$, then this extension is related to the free scalar resolvent by
\begin{equation}
\label{eqn:H0res}
(\calH_0 - (m \pm i 0))^{-1} = \begin{pmatrix}
R_0(\frac{n}{2} \mp i\tau; z,w )  &  0 \\
0 & - R_0\left(\frac{n}{2} +  \sqrt{2\mu + \tau^2}; z,w \right)   \end{pmatrix}
\end{equation}
There is an equivalent formulation for $m  = -\mu - \tau^2$.  

The hypothesis of Theorem~\ref{thm2} that $\calH$ has no embedded eigenvalues or resonances amounts to the assumption that the limiting absorption principle applies also to the perturbed resolvent, allowing us to define 
$(\calH - (m \pm i 0))^{-1}$ for $\abs{m} > \mu$.  


In the scalar case we used Stone's formula to write the spectral resolution.  This of course does not apply in the matrix case because of the lack of self-adjointness.  
However, we claim that an equivalent representation of the continuous component of the Schr\"odinger propagator still holds,
\begin{equation}\label{matrix.eith}
e^{i t \calH} P_c = \frac{1}{2 \pi i} \int_{|m| > \mu} e^{i t m} \left[ ( \calH - ( m + i0))^{-1} - (\calH - (m-i0))^{-1} \right] d m,
\end{equation}
in a suitable weak sense on weighted $L^2$ spaces.  
This representation is completely analogous to \cite[Lemma ~12]{ES1}.   Indeed, the complex analytic arguments used to establish this representation in \cite{ES1} apply directly in our case. 

From here, the proof of Theorem~\ref{thm2} works very much as in the scalar case.  We analyze \eqref{matrix.eith} using the Birman-Schwinger expansion in powers of the matrix potential $V$.  
From \eqref{eqn:H0res}, we can see that the free resolvent terms in this expansion involve either $R_0(\tfrac{n}{2} \mp i\tau)$, whose kernel was analyzed in the scalar case, or $R_0(\tfrac{n}2 + \sigma)$ 
for $\sigma>0$, whose decay properties are significantly better, as shown in \S\ref{sec:specasym}.  Hence the necessary $L^q$ estimates on the free kernels follow just as in Lemma~\ref{R0.Lp.lemma}.

The only other ingredient that we need for the matrix proof is a weighted operator norm bound on the full resolvent, analogous to 
Proposition \ref{prop:resest}.  This bound is crucial for handling the remainder term in the Birman-Schwinger expansion.  
For $m>0$ sufficiently large, we need to show
\[
\norm{\rho^{\alpha/2} \del_\lambda^q (\calH - (m \pm i 0))^{-1}\rho^{\alpha/2}}_{L^2 \to L^2} \le C_{q,\alpha} \abs{m}^{-1}.
\] 
And, just as in Proposition \ref{prop:resest}, this is a relatively simple consequence of the corresponding free bound,
\[
\norm{\rho^{\alpha/2} \del_\lambda^q (\calH_0 - (m \pm i 0))^{-1}\rho^{\alpha/2}}_{L^2 \to L^2} \le C_{q,\alpha} \abs{m}^{-1}.
\]
By \eqref{eqn:H0res}, this bound follows directly from scalar case, Proposition~\ref{R0.bounds}.

After extending these bounds to the matrix case, we can prove pointwise estimates the kernel of the operator $( \calH - ( m + i0))^{-1} - (\calH - (m-i0))^{-1}$ appearing in \eqref{matrix.eith}, just as in Proposition~\ref{ImRV.prop}.   The proof of Theorem~\ref{thm2} then follows directly by the same argument given for the scalar case in Section~\ref{sec:scalarinhom}.

\bibliographystyle{abbrv}
\bibliography{MMT-bib1}

\end{document}